\documentclass[a4paper, 12pt]{amsart}
\usepackage[nobysame]{amsrefs}
\usepackage{amsthm,enumerate, comment, mathrsfs}

\newcommand{\al}{\alpha}
\newcommand{\be}{\beta}
\newcommand{\del}{\delta}
\newcommand{\lam}{\lambda}
\newcommand{\Lam}{\Lambda}
\newcommand{\vp}{\varphi}

\newcommand{\calH}{\mathcal{H}}

\newcommand{\HbE}{\calH(E, \be)}

\newcommand{\HbS}{\calH(E, \be_S)}

\newcommand{\C}{\mathbb C}

\newcommand{\Cp}{C_\vp}

\newcommand{\sumn}{\sum_{n=1}^\infty}

\newcommand{\sumf}{\sum_{n=1}^\infty a_ne^{-\lambda_nz}}
\newcommand{\sumfone}{\sum_{n=1}^\infty a_ne^{-\lambda_nz}}

\newcommand{\limn}{\lim\limits_{n\to\infty}}
\newcommand{\limsupn}{\limsup\limits_{n\to\infty}}

\newcommand{\R}{\mathbb R}

\newcommand{\calS}{\mathcal S}

\newtheorem{thm}{Theorem}
\newtheorem{prop}[thm]{Proposition}
\newtheorem{lem}[thm]{Lemma}

\newtheorem{con}{Convention}
\newtheorem{rem}[thm]{Remark}
\newtheorem{exa}[thm]{Example}

\numberwithin{thm}{section}
\numberwithin{equation}{section}
\renewcommand{\Re}{\mathfrak{R}e}

\newcommand{\bgeqn}{\begin{equation}}
\newcommand{\edeqn}{\end{equation}}

\title[Composition Operators on Entire Dirichlet Series]{Complete Characterization of Bounded Composition Operators on the General Weighted Hilbert Spaces of Entire Dirichlet Series}

\author{Minh Luan Doan \& Le Hai Khoi}%

\address{(Doan \& Khoi) Division of Mathematical Sciences, School of Physical and Mathematical Sciences, Nanyang Technological University (NTU),
637371 Singapore}%
\address{(Doan) Current address: Department of Mathematics, University of Notre Dame, IN 46556, USA }
\email{DOAN0014@e.ntu.edu.sg;  lhkhoi@ntu.edu.sg}

\date{\today}


\thanks{Supported in part by MOE's AcRF Tier 1 grant M4011724.110 (RG128/16).}


\begin{document}
	
\maketitle

\begin{abstract}
	We establish necessary and sufficient conditions for boundedness of composition operators  on the most general class of Hilbert spaces of entire Dirichlet series with real frequencies. Depending on whether or not the space contains any nonzero constant function, different criteria for boundedness are developed. Thus, we complete the characterization of bounded composition operators on all known Hilbert spaces of entire Dirichlet series of one variable.
\end{abstract}

\section{Introduction}
	Suppose $\Lam=(\lam_n)_{n=1}^\infty$ is a sequence of real numbers that satisfies  $\lam_n\uparrow+\infty$ (i.e., $\Lam$ is unbounded and strictly increasing).  Consider a \emph{Dirichlet series with real frequencies}
	\bgeqn \label{eq:Dirichlet} \sum_{n=1}^\infty a_ne^{-\lam_nz}=a_1e^{-\lam_1z}+a_2e^{-\lam_2z}+a_3e^{-\lam_3z}+\dots, \edeqn
	where $z\in\C$ and $(a_n)\subset\C$. The series \eqref{eq:Dirichlet} is also called a \textit{general Dirichlet series}. When $\lam_n=\log n$, it becomes a \textit{classical} (or \textit{ordinary}) Dirichlet series, which has various important applications in number theory and complex analysis. If $\lam_n = n$, with the change of variable $\zeta=e^{-z}$, then \eqref{eq:Dirichlet} becomes the usual power series in $\zeta$.
	
	The classical Dirichlet series and their important role in analytic number theory are studied in the book \cite{A76}, and the theory of general Dirichlet series is presented in the excellent monograph by Hardy and Riesz \cite{HR15}. One important result from the monograph states that the region of convergence of a general Dirichlet series (if exists) is a half-plane (and for entire series, the region is the whole complex plane). Furthermore, the representation \eqref{eq:Dirichlet} is unique and holomorphic on that region of convergence.
	
	For entire Dirichlet series, Ritt \cite{R28} investigated their growth and convergence, based on which Reddy \cite{Red66} defined and formulated logarithmic orders. In the second half of the last century, Leont'ev developed theory of representation for entire functions by Dirichlet series with complex frequencies \cite{L83}. Such series are of the form \eqref{eq:Dirichlet} but with complex $\lam_n$'s.  As uniqueness no longer holds for this representation, we will not consider complex frequencies in the present article.
	
	It is clear that only finitely many elements of $\Lambda$ are negative, but there is no agreement on further restriction on the sequence. Hardy and Riesz allowed some terms $\lam_n$ to be negative. Mandelbrojt \cite{M69} supposed that all terms of $\Lambda$ are strictly positive, so nonzero constants are not representable in the form \eqref{eq:Dirichlet}. Ritt \cite{R28} allowed the possibility for free constants by adding a term $a_0$ to the series. Whether or not constants are representable by \eqref{eq:Dirichlet} affects our results in this paper, so in order to be consistent with the notations of both Mandelbrojt and Ritt, we follow the convention that $\lam_1\ge 0$, i.e., all terms of $\Lam$ are \emph{nonnegative}.

	\smallskip

	In functional analysis and operator theory, construction of Hilbert spaces of Dirichlet series and action of composition operators on them have been attractive topics for mathematicians.
	
	In the general context, let $\mathscr{H}$ be some Hilbert space whose members are holomorphic functions on a domain $G$ of the complex plane that are representable by Dirichlet series, and $\vp$ be a holomorphic self-map on $G$. The \emph{composition operator $\Cp$ acting on $\mathscr{H}$ induced by $\vp$} is defined by the rule $\Cp f=f\circ\vp,$ for $f\in\mathscr{H}$. Researchers are interested in the relation between the function-theoretic properties of $\vp$ and the operator-theoretic properties of $\Cp$. Typical problems in this topic include the invariance of $\Cp$ (i.e., $\Cp(\mathscr{H})\subseteq\mathscr{H}$), the boundedness and compactness of $\Cp$, computation of its norm and essential norms, etc.
	
	Many studies have been done on composition operators on Hilbert spaces of classical Dirichlet series. In \cite{GH99}, Gordon and Hedenmalm considered the boundedness of such operators on space of classical series with square summable coefficients. The compactness and numerical range were studied in \cite{FQ04} and \cite{FQV04}. Recently, complex symmetric composition operators have been investigated \cite{Y17}.
	
	Although entire Dirichlet series have been studied in many details, not until recently has the theory of composition operators on Banach spaces of entire Dirichlet series been developed. In \cite{HHK}, the authors proposed the construction of the general Hilbert spaces $\HbE$ of entire Dirichlet series by the use of weighted sequence spaces. Amongst the many subclasses of  $\HbE$, several properties of composition operators on them  were explored, including the boundedness, compactness and compact difference, on the most specific case, namely the spaces $\HbS$. Later, some results on essential norms of such operators \cite{HuK12}, their Fredholmness, Hilbert--Schmidtness, cyclicity and norm computation via reproducing kernels \cite{WY} on $\HbS$ were obtained.

	Specifically, let $\be$ be a sequence of positive real numbers that satisfies the following condition,
	\[\exists\alpha>0\colon  \liminf_{n\to\infty}\frac{\log\be_n}{\lam_n^{1+\al}}=+\infty.\leqno{(S)}\]
	Then the Hilbert space $\HbS$ with weight $\beta$ is defined as follow
	\[\HbS=\left\{f(z)=\sumn a_ne^{-\lam_nz}\colon \|f\|:=\Big(\sumn |a_n|^2\be_n^2\Big)^{1/2}<+\infty\right\},\]
	where the natural inner product is induced by the given norm.
	
	It is proved in \cite{HHK} that any series $f(z)=\sumn a_ne^{-\lam_nz}$ in $\HbS$ indeed represents an entire function, and such $f$ is of finite ordinary growth order.  Meanwhile, we pay attention the following theorem.
	\begin{thm}\label{T:S}
		Consider a sequence of positive real numbers $\be=(\be_n)$ that satisfies condition $(S)$ and the corresponding Hilbert space $\HbS$. A composition operator $\Cp$ induced by an entire function $\vp$, is bounded on  $\HbS$ if and only if $\vp(z) = z + b$, for some $b\in\C$ with $\Re(b) \ge 0.$
	\end{thm}

	\medskip
	
	We have two important remarks about this theorem.
	
	\emph{Firstly}, the proof given in the original paper \cite{HHK} is only applicable if $\HbS$ contains no nonzero constants (in particular, $\lam_1>0$ must hold in Proposition 4.4), while no proof was provided in the other case $\lam_1=0$. Note that the criterion for boundedness of $\Cp$ will be different if $\HbS$ contains nonzero constants. For instance, any constant $\vp$ will now induce a bounded operator $\Cp$, so Theorem \ref{T:S} has not covered all possibilities.
	
	\emph{Secondly}, we note that the proof of the theorem strongly relies on the following lemma.
	
	\begin{lem}[{P\'olya \cite{Pol}}]\label{L:Pol}
		Let $g$ and $h$ be entire functions such that $f=g\circ h$ is of finite (ordinary) order. Then either
		\begin{enumerate}[(i)]
			\item $h$ is a polynomial and $g$ is of finite order, or
			\item $h$ is not a polynomial, but an entire function of finite order, and $g$ is of order $0$.
		\end{enumerate}
	\end{lem}

	In order to use Lemma \ref{L:Pol}, orders of entire functions in the space must be finite, so condition $(S)$ is imposed on the weight sequences $\be$  of the induced spaces $\HbS$. In addition, we highlight that all the aforementioned results of $\Cp$ in \cite{HHK,HuK12,WY} are established only for spaces $\HbS$, due scope of the known proof of Theorem \ref{T:S}. Because the first and most important property is the boundedness, and other problems such as compactness, compact difference, etc. can only be resolved thereafter, we must find a \textbf{new approach} to establish the boundedness of $\Cp$ that does not involve Lemma~\ref{L:Pol} when dealing with spaces that are more general than $\HbS$. As far as we know, there has been no successful answer to this problem.
	
	\medskip
	
	Therefore, a natural question can be asked is: \emph{what are the criteria for boundedness of composition operators on some Hilbert spaces of entire functions that belong to a class that contains spaces $\HbS$ as special cases?}
	
	The aim of this research article is to tackle the proposed question. We will work with the spaces $\HbE$, the \emph{most general} class of Hilbert spaces of entire Dirichlet series that we know up to now.  Thus, we provide a \emph{complete characterization of the boundedness of composition operators $\Cp$}.
	
		As we will see later, Lemma \ref{L:Pol} fails to be applied to the general spaces $\HbE$.	Hence, we propose \textbf{different techniques} of proof from that of \cite{HHK}, which covers both cases in the first remark above. We note that the criteria in those cases are not identical, and their proofs are not trivial applications of each other.
	
	The structure of the paper is as follows. We provide in Section \ref{sec:prelim} a summary of known results about Hilbert spaces of entire Dirichlet series, most importantly the construction of spaces $\HbE$. Section \ref{sec:RKHS} presents important notions of reproducing kernels on spaces $\HbE$, which is helpful for subsequent sections. In Section \ref{sec:HbE}, we deal with boundedness of composition operators. In particular, we first propose a sufficient condition in Proposition \ref{P:suf}, and later prove that this condition is also necessary. In Subsections \ref{sec:suf-cond} and \ref{sec:nec-cond}, boundedness of $\Cp$ for the most general class $\HbE$ is studied, in both cases when a space $\HbE$ does not contain nonzero constants (Theorem \ref{T:bounded1}) and when it does (Theorem \ref{T:bounded2}). A summary of our results and some concluding remarks are given in Section \ref{sec:conclusion}.
	
	\section{Hilbert spaces $\HbE$ of entire Dirichlet series}\label{sec:prelim}
	For a given sequence $\Lam=(\lam_n)_{n=1}^\infty$ with $0\le \lam_n\uparrow +\infty$, define the following constant $L$,
	\[L:=\limsupn \frac{\log n}{\lam_n}.\]
	
	We associate to each Dirichlet series \eqref{eq:Dirichlet} the following quantity,
	\[D:=\limsupn \frac{\log|a_n|}{\lam_n}.\]
	
	It is well-known that $L$ is the upper bound of the distance between the abscissa of convergence and the abscissa of absolute convergence of the series \eqref{eq:Dirichlet}. We refer the reader to \cite{HR15} for the basic properties of these abscissas. If ${L<+\infty}$, then the Dirichlet series \eqref{eq:Dirichlet} represents (uniquely) an entire function if and only if $D=-\infty$ (see, e.g., \cite{HK12,M69}).
	
	\begin{con}\label{Conv:1}
		Throughout this paper, the condition $L<+\infty$ is always supposed to hold.
	\end{con}
	
	Now, let $\be=(\be_n)$ be a sequence of (not necessarily distinct or monotonic) positive numbers. We introduce the following \emph{weighted sequence space with weight $\be$}:
	\[\ell_\be^2=\left\{ {\bf a}=(a_n)_{n=1}^\infty\subset\C\colon\|{\bf a}\|_{\ell^2_\beta}=\Big(\sumn |a_n|^2\be_n^2\Big)^{1/2}<+\infty \right\},\]
	which is a Hilbert space with the inner product of any ${\mathbf a}=(a_n)$ and ${\mathbf b}=(b_n)$ in $\ell^2_\be$ given by
	$$\langle {\bf a},{\bf b} \rangle_{\ell^2_\be}=\sumn a_n\overline{b_n}\be_n^2.$$
	
	The sequence spaces $\ell^2_\beta$ have an important role in the construction of many important Hilbert spaces by varying $\be$, such as Hardy spaces, Bergman spaces, Dirichlet spaces, Fock spaces, etc. (see, e.g., the book \cite{CM}).
	
	Consider the following \emph{function space $\calH(\be)$ of entire Dirichlet series induced by weight $\be$}:
	\bgeqn\label{eq:HD}
	\calH(\be)=\left\{f(z)=\sumfone\textnormal{ entire}\colon\|f\|_{\calH(\be)}:=\|(a_n)\|_{\ell^2_\beta}<+\infty \right\}.
	\edeqn
	Here, when we write $f(z)=\sumfone$, we mean the entire function $f$ is represented by the series on the right-hand side.

	The space $\calH(\be)$ is an inner product space, where
	\[\langle f,g\rangle_{\calH(\be)}=\sumn a_n\overline{b_n}\be_n^2,\]
	for any $f(z)=\sumfone$ and $g(z)=\sumn b_ne^{-\lam_nz}$ in $\calH(\be).$
	
	\medskip
	
	Depending on $\be$, the induced space $\calH(\be)$ may not be complete in its norm, and so it is not necessarily a Hilbert space. The following theorem  from \cite{HHK} provides a criterion of the weight $\be$ for $\calH(\be)$ to be complete.
	\begin{thm}  \label{T:Hilbert} The space $\calH(\be)$ of entire Dirichlet series induced by a sequence of positive real numbers $\be=(\be_n)$, as defined in \eqref{eq:HD}, is a Hilbert space if and only if the following condition $(E)$ holds,
		$$
		\label{eq:Hilbert} \liminf_{n\to\infty}\frac{\log\be_n}{\lambda_n}=+\infty.\leqno{(E)}
		$$
	\end{thm}

	A direct consequence of this theorem is that if $(E)$ holds, the space $\calH(\be)$ automatically becomes a Hilbert space of entire functions, so  we can drop the condition ``entire" in  \eqref{eq:HD}.
	
	Note that when $(E)$ holds, if $0\in\Lam$, i.e., $\lam_1=0$, then the space contains all nonzero constants, while it contains no nonzero constants if $\lam_1>0$. Obviously, Theorem \ref{T:Hilbert} is unaffected regardless $\lam_1$ is $0$ or not. Hence, we adopt the following convention.
	
	\begin{con} Unless otherwise stated, we assume condition $(E)$ always holds. We denote by $\HbE$ the following Hilbert space of entire Dirichlet series
		\[\HbE=\left\{f(z)=\sumfone\ :\ \|f\|_{\HbE}=\Big(\sumn |a_n|^2\be_n^2\Big)^{1/2}<+\infty\right\},\]
		and without ambiguity, we denote the norm of any function $f\in\HbE$ simply by $\|f\|$.
	\end{con}

	\medskip
	
	Suppose $f$ is an entire function. The \emph{ordinary growth order} of $f$ is the limit
	\[\rho=\limsup_{r\to\infty}  \frac{\log(\log \|f\|_{\infty,r})}{\log r},\]
	where $\|f\|_{\infty,r}=\sup_{|z|\le r}|f(z)|$ $(r\ge0)$.
	
	If the series \eqref{eq:Dirichlet} represents an entire function $f$, the \emph{Ritt order $\rho_R$} of $f$ is defined to be the limit
	\[\rho_R:=\limsupn\dfrac{\lam_n \log \lam_n}{\log\frac{1}{|a_n|}}.\]
	Ritt orders of entire Dirichlet series are studied in \cite{R28}.
	
	Suppose in addition that $f$ has Ritt order $0$, write $z=\sigma+ti\ (\sigma,t\in\R)$, Reddy \cite{Red66} defined the \emph{logarithmic orders} of $f$ as follows:
	\begin{align*}
		\rho_R(\mathfrak L):&=\limsup_{\sigma\to\infty} \dfrac{\log \log M(\sigma)}{\log(-\sigma)},\\
		\rho_*(\mathfrak L):&=\limsup_{\sigma\to\infty} \dfrac{\log \log \mu(\sigma)}{\log(-\sigma)},\\
		\rho_c(\mathfrak L):&=\limsupn \dfrac{\log \lam_n}{\log\left(\frac{1}{\lam_n}\log \frac{1}{|a_n|}\right)},
	\end{align*}
	where $M(\sigma) = \sup_{t\in\R}| f(\sigma + it)|$ and $\mu(\sigma)=\max_{n\ge 0} \{|a_n|e^{-\lam_n\sigma}\}$. He also showed that $$\rho_R(\mathfrak L)=\rho_*(\mathfrak L)=\rho_c(\mathfrak L)+1\ge 1.$$

	The lemma below explains a correspondence between the space $\HbE$ and the growth orders of its elements.
	\begin{lem}[{\cite{HHK}}]\label{L:fin-ord1} Let $\be$ be a sequence of positive real numbers. Then \emph{every} element of $\HbE$ represents an entire function with finite logarithmic orders if and only if the following condition holds,
	\[\exists\al>0:\ \liminf_{n\to\infty}\frac{\log\be_n}{\lam_n^{1+\al}}=+\infty.\eqno{(S)}\]
	\end{lem}
	
	If $(S)$ holds, the Hilbert space is denoted  by $\HbS$ in \cite{HHK}. Clearly, condition $(S)$ is stronger than condition $(E)$, thus spaces $\HbS$ are special cases of the general class $\HbE$. We note the following relation between logarithmic orders and ordinary orders.

	\begin{lem}[\cite{HK12}]\label{L:fin-ord2} Every entire Dirichlet series of finite logarithmic orders has finite (ordinary) order.
	\end{lem}
	
	Lemmas \ref{L:fin-ord1} and \ref{L:fin-ord2} imply that every element of a space $\HbS$ is an entire function of finite order, which explains why Lemma \ref{L:Pol} was used in \cite{HHK} to derive a criterion of bounded composition operators on $\HbS$.
	
	Nevertheless, the space $\HbS$ is quite small, in the sense of Example~\ref{ex:noHbS} below. In fact, the class of $\HbS$ is the smallest class considered in \cite{HHK}.
	
	\begin{exa}\label{ex:noHbS}
		Let $\lam_n=n$. Clearly $L=0$. Consider the entire function $f(z)=e^{e^{-z}}-1$. We can verify that $f$ has infinite growth order, so Lemma \ref{L:fin-ord2} implies that there is no weight $\be$ satisfying $(S)$ such that $f$ is representable by series in the induced space $\HbS$.
		
		What about the existence of a space $\HbE$ that contains $f$? The answer is positive. Consider $\be_n=\sqrt{n!}$, we can verify that $\be=(\be_n)$ satisfies $(E)$. We have
		$$f(z)=e^{e^{-z}}-1=\sumn \frac{e^{-nz}}{n!}=\sumn\frac{1}{n!}e^{-\lam_nz}$$
		So $\|f\|^2=\sumn (n!)^{-1}=e-1<+\infty$. This shows $f$ belongs to the space $\HbE$ induced by $\be$.  		
	\end{exa}

	Since we are working with the general class $\HbE$, from now on, we do not need any results about $\HbS$.
	
	\section{Reproducing kernel Hilbert spaces $\HbE$}\label{sec:RKHS}
	
	A (complex) separable Hilbert space $\mathscr{H}$ of functions from a non-empty set $G\subseteq\C$ to $\C$ is called a \emph{reproducing kernel Hilbert space} (RKHS) if for every $y\in G$, the evaluation functional $\del_y:f\mapsto f(y)$ ($f\in \mathscr{H}$) is bounded.
	
	By Riesz Representation Theorem, there exists a unique element $k_y\in \mathscr{H}$ such that  $f(y)=\langle f,k_y\rangle_{\mathscr{H}}$ for every $f\in \mathscr{H}$. We call $k_y$ the \emph{reproducing kernel at the point} $y$.
	
	The function $K:G\times G\to \C$ defined by
	$$K(x,y)=\langle k_y,k_x\rangle_{\mathscr{H}}=k_y(x),\quad x,y\in \mathscr{H},$$
	is called the \emph{reproducing kernel for} $\mathscr{H}$. It is well known that if a collection of elements $\{e_j\}_{j=1}^\infty$ is an orthonormal basis for $\mathscr{H}$, then
	\bgeqn \label{eq:RKHSformula} K(x,y) = \sum_{j=1}^\infty e_{j}(x)\overline{{e}_j(y)},\edeqn
	where the convergence is pointwise for $x,y\in \mathscr{H}$ (see the famous article \cite{Ar50}).
	
	\medskip
	
	We show in the following proposition that if all elements of $\HbE$ are entire Dirichlet series, i.e., if $\be$ satisfies $(E)$, then $\HbE$ is a reproducing kernel Hilbert space.
	\begin{prop}\label{RKHS}
		Let $\be=(\be_n)$ satisfy condition $(E)$. Then the space $\HbE$ induced by $\be$ is a complex reproducing kernel Hilbert space with the reproducing kernel $K:\C\times \C\to\C$ given by
		\bgeqn K(z,w)=k_w(z)=\sum_{n=0}^{\infty}\frac{e^{-\lam_n(\overline{w}+z)}}{\be_n^2}.\label{eq:RKHS}\edeqn
		The convergence is uniform on compact subsets of $\C\times\C$.
	\end{prop}
	
	\begin{proof}
		Apply Cauchy--Schwarz inequality, we have
		\begin{align*}
			|f(z)|^2=\left|\sumf\right|^2&\le\left(\sumn \frac{e^{-2\lam_n\Re(z)}}{\be_n^2}\right)\left(\sumn|a_n|^2\be_n^2\right)\\
			&=\left(\sumn \frac{e^{-2\lam_n\Re(z)}}{\be_n^2}\right)\|f\|^2=M_z\|f\|^2.
		\end{align*}
		
		Note that since $\lim_{n\to\infty}\lam_n^{-1}\log \be_n=+\infty$, the series $M_z$ is convergent absolutely for any $z\in\C$. Hence for each complex $z$, there exists a corresponding constant $M_z>0$ such that $|f(z)|^2\leq M_z\|f\|^2$ for all $f\in\HbE$. Each evaluation functional $\del_z$ is thus bounded, which shows that $\HbE$ is an RKHS.
		
		We can verify that the \emph{probe functions} $q_n(z)=\be_n^{-1}e^{-\lam_n z}\ (n\ge1)$ forms an orthonormal basis of $\HbE$. From \eqref{eq:RKHSformula}, we have
		\begin{align*}
			K(z,w) & = \sum_{n=1}^\infty \be_n^{-1}e^{-\lam_n z}\be_n^{-1}e^{-\lam_n \overline{w}}=\sumn \be_n^{-2}e^{-\lam_n (z+\overline{w})}.
		\end{align*}
	
		Finally, consider $K(z,w)=k_w(z)$ as a Dirichlet series in variable $z$ and coefficients $a_n=\be_n^{-2}e^{-\lam_n\overline{w}}$. We can derive from condition $(E)$ that
		\[D=\limn \frac{\log |a_n|}{\lam_n}=\limn \left(-\overline{w}-2\frac{\log \be_n}{\lam_n}\right)=-\infty.\]
		
		By the discussion before Convention \ref{Conv:1}, the series converges absolutely on compact sets of $z$. We obtain the similar result if we exchange the role of $z$ and $w$. By the uniform convergence on compact sets of $\C$ for each variable, $K$ is uniformly convergent on compact subsets of $\C^2$. The proof is complete.
	\end{proof}
	
	\begin{rem}\hfil\label{R:RKHS}
		\begin{enumerate}[(a)]
			\item In the proof above, we can easily see that for any $w\in \C$,
			$$\|k_w\|^2=K(w,w)=\sumn \frac{e^{-2\lam_n\Re(w)}}{\be_n^2}.$$
			
			\item 	By a consequence of closed graph theorem, if a composition operator $\Cp$ is \textit{invariant}, that is, 
			if $\Cp( \mathscr{H})\subseteq  \mathscr{H}$, then it is automatically bounded. Thus, we don't have to deal with invariance and boundedness separately, since the two properties are equivalent for $\Cp$ acting on RKHSs.
		\end{enumerate}
	\end{rem}
	
	\section{Main results}\label{sec:HbE}
	In the sequel, we fix a sequence $\be=(\be_n)$ that satisfies $(E)$ and let $\HbE$ be the corresponding Hilbert space of Dirichlet series.
	
	We remind an important point, which is seen later, that the criteria for boundedness of $\Cp$ for the case $\lam_1>0$ and for the case $\lam_1=0$  are different. Recall that if $\lam_1=0$, the space $\HbE$ also includes all constants, and that the space contains no nonzero constants if ${\lam_1>0}$. The proof of the necessary condition in the latter case is also more sophisticated than the former, even though the idea used in the two proofs are similar. This fact is reflected in Propositions \ref{P:lam>0} and \ref{P:lam=0}.
	
	\subsection{Sufficient conditions}\label{sec:suf-cond}\hfill
	
	We can easily obtain the following sufficiency for the boundedness of $\Cp$ on $\HbE$.
	
	\begin{prop} \label{P:suf}
		Let $\vp$ be an entire function. Consider the statements below.
		\begin{enumerate}[(i)]
			\item $\vp$ is a constant function,
			\item $\vp(z)=z+b$ for some $b\in\C,\ \Re(b)\ge0$.
		\end{enumerate}
	
		The following are true:
		\begin{enumerate}[(a)]
			\item Suppose $\lam_1=0$. If either $(i)$ or $(ii)$ holds, then $\Cp$ is bounded.
			\item Suppose $\lam_1>0$. If $(ii)$ holds, then $\Cp$ is bounded.
		\end{enumerate}		
	\end{prop}

	\begin{proof}
		$\bullet$ Note that the difference between $(a)$ and $(b)$ is that the case ``$\vp$ is a constant function" is not included when $\lam_1>0$. This can be seen as follows. Take, for instance, $f(z)=e^{-\lam_1z}\in\HbE$. If $\vp(z)=z_0$ for all $z\in\C$, then $\Cp f(z)=e^{-\lam_1z_0}$, which is a nonzero constant, and thus not representable in $\HbE$ if $\lam_1>0$. 	
		
		$\bullet$ Suppose $\lam_1=0$. Clearly if $(i)$ happens, i.e., $\vp(z)=z_0$ $(z\in\C)$ for some $z_0\in\C$, then
		\[\|\Cp f\|=\|f(z_0)\|=|f(z_0)|\be_1\le \be_1\|K_{z_0}\|\|f\|,\]
		by Cauchy--Schwarz inequality. Hence, $\Cp$ is bounded in this case.
		
		$\bullet$ We will use the following argument to prove that $(ii)$ implies ``$\Cp$ is bounded" in both cases $\lam_1>0$ and $\lam_1=0$.	
	
		Suppose $(ii)$ holds, we have
		$$\Cp f(z)=\sumn a_ne^{-\lam_n(z+b)}=\sumn a_ne^{-\lam_nb} e^{-\lam_nz},$$
		for any $f(z)=\sumf\in\HbE$.
		
		Since $\Re(b)\ge0$ and $(\lam_n)$ is increasing, we have
		\begin{align}\label{eq:norm1}\|\Cp f\|^2&=\sumn |a_n|^2e^{-2\lam_n\Re(b)}\be_n^2\notag\\&\le e^{-2\lam_1\Re(b)}\sumn |a_n|^2\be_n^2=e^{-2\lam_1\Re(b)}\|f\|.\end{align}
		
		This shows $\Cp$ is bounded. The proof is complete.
	\end{proof}
	
	\subsection{Necessary conditions}\label{sec:nec-cond}
	The sufficient conditions in Proposition \ref{P:suf} turn out to be necessary as well. Our aim is to establish the proof for this necessity.
	
	\smallskip
	
	The following lemma is needed for next results. An analogous version of this lemma can be found in \cite{Red66}, but we also provide a proof here for the sake of completeness. 
	\begin{lem}\label{L:coeff}
		Suppose $f\in\HbE$ has the representation
		\[f(z)=\sumf\qquad (a_n,z\in\C).\]
		Then for any $\sigma\in\R$, for any $n\ge1$,
		\bgeqn a_n=\lim_{t\to+\infty}\frac{1}{2ti}\int_{\sigma-ti}^{\sigma+ti} f(z)e^{\lam_nz}dz,\label{eq:coeff}\edeqn
		where the integral is taken on the line segment from $\sigma-ti$ to $\sigma+ti$
	\end{lem}
	
	\begin{proof}
		Fix a particular $n$. Define $\mu_k=\lam_n-\lam_k$. Multiply both sides of $f$ by $e^{\lam_nz}$, we have
		\bgeqn\label{eq:new} f(z)e^{\lam_nz}=a_1e^{\mu_1z}+a_2e^{\mu_2z}+a_3e^{\mu_3z}+\dots\edeqn
		
		For any $\sigma\in\R$ and $t>0$, we integrate both sides of \eqref{eq:new} on the line segment from $\sigma-ti$ to $\sigma+ti$. Since $f(z)e^{\lam_nz}$ is uniformly convergence for all $z$, we can integrate term by term on the right-hand side to obtain
		\bgeqn\label{eq:int}\int_{\sigma-ti}^{\sigma+ti} f(z)e^{\lam_nz}dz = \sum_{k=0}^\infty a_n \int_{\sigma-ti}^{\sigma+ti} e^{\mu_kz} dz.\edeqn
		
		Note that for any $\mu\in\R$,
		\bgeqn\frac{1}{2ti}\int_{\sigma-ti}^{\sigma+ti} e^{\mu z} dz
		=\begin{cases}1\quad&\hbox{if $\mu=0$,}\\
		\dfrac{e^{\mu \sigma}}{\mu t} \sin(\mu t)&\hbox{if $\mu\ne0$}.
		\end{cases}\notag\edeqn
		
		Thus, \eqref{eq:int} is equivalent to
		\[\frac{1}{2ti}\int_{\sigma-ti}^{\sigma+ti} f(z)e^{\lam_nz}dz=a_n+\sum_{k\ne n} a_k\dfrac{e^{\mu_k \sigma}}{\mu_k t}\sin(\mu_k t).\]
		
		Letting $t\to\infty$ on both sides, and taking into account the uniform convergence of the series on the right-hand side, we obtain \eqref{eq:coeff}.
	\end{proof}

	We also need the following familiar fact.
	\begin{lem}\label{adj}
		Suppose a composition operator $\Cp$, induced by an entire function $\vp$, maps $\HbE$ to itself. Then the adjoint operator $\Cp^*$ of $\Cp$ satisfies
		$$\Cp^*k_w=k_{\vp(w)},\quad \forall w\in\C,$$
		where $k_w$ is the reproducing kernel at $w$ as defined in \eqref{eq:RKHS}.
	\end{lem}

	\subsubsection{Case $\lam_1>0$}\label{sec:>0}
	\hfill
	
	We have the following necessary condition:
	
	\begin{prop}\label{P:lam>0}
		Suppose $\lam_1>0$. Let $\vp$ be an entire function and $\Cp$ be the induced composition operator. If $\Cp$ is bounded on $\HbE$, then
		\[\vp(z)=z+b,\hbox{ with $b\in\C$, $\Re(b)\ge 0$}.\]
	\end{prop}
	
	\begin{proof}
		Suppose $\Cp$ is bounded on $\HbE$, then its adjoint operator $\Cp^*$ is also bounded. That is, there is a constant $B>0$ such that
		\bgeqn\|\Cp^*f\|^2\le B\|f\|^2,\qquad\forall f\in\HbE.\label{eq:ineq}\edeqn
		Without the loss of generality, we may assume $B>1$.		
		
		In particular, for $f=k_w$ where $w$ is an arbitrary complex number, we note that $\Cp^*k_w=k_{\vp(w)}$, so together with Remark \ref{R:RKHS} (a), the inequality \eqref{eq:ineq} becomes
		\bgeqn
		\label{eq:ineqref}\sumn \be_n^{-2}e^{-2\lam_n\Re(\vp(w))}\le B\sumn \be_n^{-2}e^{-2\lam_n\Re(w)},\qquad\forall w\in\C.
		\edeqn
		
		\noindent$\bullet$ \emph{\underline{Claim 1:} We have $\vp(z)=z+b$ for some $b\in\C$.}
		
		Assume $\psi(z):=z-\vp(z)$ is a non-constant entire function, we show the contradiction by finding some $w\in\C$ such that inequality \eqref{eq:ineqref} does not hold.
		
		Since $\psi$ is not a constant function, the function $F(w)=e^{\lam_1\psi(w)}$ is also a non-constant entire function. By Liouville's theorem, $F$ is not bounded, so we can choose a fixed $w=w_0\in\C$ so that
		\[|F(w_0)|^2=e^{2\lam_1\Re(\psi(w_0))}\ge 2B >1.\]
		
		This implicitly means $\Re(\psi(w_0))> 0$. Noting that $(\lam_n)$ is increasing, from Remark \ref{R:RKHS} (a), we have
		\begin{align}
			\|k_{\vp(w_0)}\|^2=\sumn \be_n^{-2}&e^{-2\lam_n\Re(\vp(w_0))}= \sumn e^{2\lam_n\Re(\psi(w_0))}\be_n^{-2}e^{-2\lam_n\Re(w_0)}\notag\\
			&\ge |F(w_0)|^2 \sumn \be_n^{-2}e^{-2\lam_n\Re(w_0)} \label{eq:k2}\\
			&\ge 2B \sumn \be_n^{-2}e^{-2\lam_n\Re(w_0)}>B\|k_{w_0}\|^2\notag,
		\end{align}
		which clearly contradicts the inequality \eqref{eq:ineqref}. Thus, $\vp(z)=z+b$ for some $b\in\C$.
	
		\smallskip
		\noindent $\bullet$ \emph{\underline{Claim 2}: We have $\Re(b)\ge0$}.
		
		Consider the probe functions $q_k(z)=\be_k^{-1}e^{-\lam_kz}\ (k\ge1)$ introduced in Section \ref{sec:RKHS}. Since $\Cp$ is bounded and $\|q_k\|=1$, the sequence $(\|\Cp q_k\|)_k$ must be bounded. We note that $\Cp q_k(z)=\be_k^{-1}e^{-\lam_k(z+b)}$, so
		\bgeqn\|\Cp q_k(z)\|=e^{-\lam_k\Re(b)}\label{eq:norm}.\edeqn
		
		Since $\lam_k\uparrow+\infty$, it is necessary that $-\Re(b)\le 0$, i.e., $\Re(b)\ge 0$.
		
		The proof is complete.
	\end{proof}

	\begin{rem}
		Proposition \ref{P:lam>0} is similar to the necessity of Theorem \ref{T:S}. To obtain this result in \cite{HHK}, the authors first proved that the function $\vp$ necessarily has the form $az+b$, then derived two other lemmas, before eventually showed that $a=1$. This proof strongly depends on Lemma \ref{L:Pol} and long. Our approach is much simpler, which is applicable to the general spaces $\HbE$ and only utilizes fundamental results of functional analysis.
	\end{rem}

	Now we obtain the following criterion for the bounded composition operators in the case $\lam_1>0$.

	\begin{thm}[Criterion for $\lam_1>0$]\label{T:bounded1}
		Suppose $\lam_1>0$. Let $\vp$ be an entire function and $\Cp$ be the induced composition operator. Then the composition operator $\Cp$ is bounded on $\HbE$ if and only if
		$$\vp(z)=z+b, \hbox{ for some $b\in\C$ with $\Re(b)\ge 0$}.$$
		
		Moreover, the operator norm is given by $\|\Cp\|=e^{-\lam_1\Re(b)}$.
	\end{thm}
	
	\begin{proof} The necessary condition is proved in Proposition \ref{P:lam>0}, while the sufficiency is shown in Proposition \ref{P:suf}. Thus, $\Cp$ is bounded if and only if $\vp(z)=z+b$ for some $b\in\C$ with nonnegative real part.
		
		To compute the norm of $\Cp$, note that \eqref{eq:norm} implies
		\bgeqn\label{eq:norm2}\|\Cp\|\ge \|\Cp q_1\|=e^{-\lam_1\Re(b)}.\edeqn
		
		From \eqref{eq:norm1} and \eqref{eq:norm2}, we obtain $\|\Cp\|=e^{-\lam_1\Re(b)}$.
	\end{proof}

	Since spaces $\HbS$ are special cases of spaces $\HbE$, we easily recover Theorem \ref{T:S}.
	
	\subsubsection{Case $\lam_1=0$}\label{sec:=0}
	\hfil
	
	To establish the necessity for the boundedness of composition operators on $\HbE$, we again use the adjoint operator $\Cp^*$, but with the approach that is more complicated than that of Theorem \ref{T:bounded1}. The difference comes from the fact that if $\lam=1$, then $F(w)=1$ in the proof of Theorem \ref{T:bounded1}, and so we do not have \eqref{eq:k2}. One might attempt to introduce ${F(w)=e^{\lam_2\psi(z)}}$, but still the first inequality of \eqref{eq:k2} is not true. Hence, a nontrivial adjustment is necessary.
	
	
	\begin{prop}\label{P:lam=0}
		Suppose $\lam_1=0$. Let $\vp$ be an entire function and $\Cp$ be the induced composition operator. If the operator $\Cp$ is bounded on $\HbE$, then exactly one of the following possibilities happens:
		\begin{enumerate}[$(i)$]
		\item $\vp$ is a constant function, or
		\item $\vp(z)=z+b,$ for some $b\in\C$ with $\Re(b)\ge 0$.
		\end{enumerate}
	\end{prop}
	
	\begin{proof} Suppose $\Cp$ is bounded. Since $\lam_1=0$, we have $\lam_2>0$.
		
		If $\vp(z)=z+b$ for some $b\in\C$, we obtain condition $\Re(b)\ge 0$ in the same way as in Claim 2 of Theorem \ref{T:bounded1}.
		
		If $\vp$ is not of the form $z+b$, we prove that $\vp$ must be constant, then the proof is complete.
		
		\medskip
		
		Since $\Cp$ is bounded, so is the adjoint operator $\Cp^*$. Hence, there is a constant $B>1$ such that $$\|\Cp^* f\|\le B\|f\|\quad\forall f\in\HbE.$$
		
		Since $\vp$ is not of the form $z+b$, the function $\psi(z)=z-\vp(z)$ is not constant. Thus, the function $Q(z)=e^{\lam_2\psi(z)}$ is entire and not constant either. By Liouville's theorem, $Q$ is not bounded, i.e., there exists $(z_k)\subset \C$ such that $|Q(z_k)|\to\infty$ as $k\to\infty$. This allows us to define the following nonempty set of sequences:
		$$\calS:=\left\{(z_k)\subset\C: \lim_{k\to\infty} |Q(z_k)|=+\infty\right\}$$
		
		From this point, our proof is divided into several claims as follows.
		
		$\bullet$ \emph{\underline{Claim 1:} If $(z_k)\in \calS$, then $(\Re (z_k))$ is not bounded above}.
		
		Assume there is a sequence $(z_k)\in \calS$ such that $\Re(z_k)<T$ for some $T>0$. As $|Q(z_k)|\to\infty$, there exists some $w_0\in(z_k)$ such that
		\bgeqn\label{eq:k}|Q(w_0)|^2=e^{2\lam_2\Re(\psi(w_0))}>B+\frac{\be_2^2(B-1)}{\be_1^2e^{-2T\lam_2}}>B.\edeqn
		
		Note that $\Re(\psi(w_0))>0$ is implicitly implied in the inequality above, as $|Q(w)|>B>1$. From \eqref{eq:k} we have
		 \[\frac{e^{-2\lam_2\Re(w_0)}}{\be_2^2}\left(e^{2\lam_2\Re(\psi(w_0))}-B\right)>\frac{e^{-2\lam_2T}}{\be_2^2}\left(e^{2\lam_2\Re(\psi(w_0))}-B\right)>\frac{B-1}{\be_1^2}.\]
		
		Substituting back $\psi(w_0)=w_0-\vp(w_0)$, we obtain
		\bgeqn\frac{e^{-2\lam_2\Re(\vp(w_0))}}{\be_2^2}>\frac{B-1}{\be_1^2}+\frac{Be^{-2\lam_2\Re(w_0)}}{\be_2^2}.\label{eq:ineq2}\edeqn
		
		Since $\lam_n\Re(\psi(w_0))>\lam_2\Re(\psi(w_0))>B$ for all $n>2$, inequality \eqref{eq:ineq2} implies
		\begin{align*}
		\|k_{\vp(w_0)}\|^2&=\frac{1}{\be_1^2}+\frac{e^{-2\lam_2\Re(\vp(w_0))}}{\be_2^2}+\sum_{n=3}^\infty \frac{e^{-2\lam_n\Re(\vp(w_0))}}{\be_n^2}\\
		&>\frac{1}{\be_1^2}+\frac{B-1}{\be_1^2}+\frac{Be^{-2\lam_2\Re(w_0)}}{\be_2^2}+\sum_{n=3}^\infty e^{2\lam_n\Re(\psi(w_0))}\frac{e^{-2\lam_n\Re(w_0)}}{\be_n^2}\\
		&>\frac{B}{\be_1^2}+B\sum_{n=2}^\infty \frac{e^{-2\lam_n\Re(w_0)}}{\be_n^2}=B\|k_{w_0}\|^2.
		\end{align*}
		
		Again, inequality \eqref{eq:ineqref} does not hold, and we obtain a contradiction. Thus, every sequence $(z_k)\in \calS$ has no upper bound.
		\medskip
		
		$\bullet$ \emph{\underline {Claim 2:} The function $Q$ is bounded on the half-plane $\Re(z)<0$}.
		
		Assume $Q$ is unbounded on the half-plane $\Re(z)<0$, then there exists a sequence $(z_k)\subset\C$ such that $\Re(z_k)<0$ and $|Q(z_k)|\to\infty$ as $k\to\infty$. Hence, $(z_k)\in \calS$ and $(\Re (z_k))$ is bounded above. This clearly contradicts Claim 1.
		
		\medskip
		
		$\bullet$  \emph{\underline {Claim 3:} We have the representation $e^{-\lam_2\vp(z)}=a_1+a_2e^{-\lam_2z}$ for some $a_1,a_2\in\C$.}
		
		From Claim 2, there exists some $M>0$ such that $|Q(z)|<M$, if ${\Re(z)<0}$. Substituting $\vp(z)=z-\psi(z)$, we have
		\bgeqn|M^{-1}e^{-\lam_2\vp(z)}|<e^{-\lam_2\Re(z)}, \qquad \textnormal{for all $z$ with } \Re(z)<0\label{eqn:bounded}.\edeqn
		
		Consider the function $f(z)=e^{-\lam_2z}\in\HbE$. Since $\Cp$ maps $\HbE$ to itself, we have
		\[\Cp f(z) = e^{-\lam_2\vp(z)}=\sum^\infty_{n=1}a_ne^{-\lam_nz},\]
		for some $(a_n)\subset\C$.
		
		Divide each expression of the equality above by $M$, we obtain
		\bgeqn\label{eqn:equality} g(z)=M^{-1}e^{-\lam_2\vp(z)}=c_1+c_2e^{-\lam_2z}+c_3e^{-\lam_3}+\dots,\edeqn
		where $c_n=a_n/M$.
		
		From \eqref{eqn:bounded} and \eqref{eqn:equality}, it follows that $|g(z)|<e^{-\lam_2\Re(z)}$ for all $z$ with ${\Re(z)<0}$. For any $n>2$, we write $z=\sigma+ti\ (\sigma,t\in\R)$ and apply Lemma~\ref{L:coeff} to get
		\[|c_n|=\left|\lim_{t\to\infty} \frac{1}{2ti}\int_{\sigma-ti}^{\sigma+ti}g(z)e^{\lam_nz}dz\right|\le e^{\sigma(\lam_n-\lam_2)}.\]
		
		As the inequality above is true for any $\sigma\in\R$, we have
		$$|c_n|=\lim_{\sigma\to-\infty} e^{(\lam_n-\lam_2)\sigma}=0,\quad \hbox{for all $n>2$}.$$
		
		Thus $a_n=0$ for $n>2$. From the uniqueness of the representation of $e^{-\lam_2\vp(z)}$, we have
		\bgeqn e^{-\lam_2\vp(z)}=a_1+a_2e^{-\lam_2z},\label{eqn:equation}\qquad\forall z\in\C.\edeqn
		
		$\bullet$ \emph{\underline {Claim 4:} The function $\vp$ is constant.}
		
		With the same notation as in Claim 3, we have the following cases:
		\begin{enumerate}[(i)]
		\item If $a_2\ne0$ and $a_1\ne0$: the right hand side of \eqref{eqn:equation} is zero at
		$$z=-\lam_{2}^{-1}\left(\ln\left|\frac{a_1}{a_2}\right|+i\textnormal{Arg}\frac{a_1}{a_2}+i(2k+1)\pi\right)\ (k\in\mathbb Z),$$ while the left hand side function is never zero, so we obtain a contradiction. This shows $a_1$ and $a_2$ cannot be both nonzero.
		
		\item If $a_2\ne0$ and $a_1=0$: equation \eqref{eqn:equation} implies $$\vp(z)=z-\lam_2^{-1}(\ln|a_2|+i(\textnormal{Arg}\ a_2+2k\pi)),$$ for some $k\in\mathbb Z$, which contradicts the assumption $\psi$ is not constant. This shows $a_2=0$.
		
		\item If $a_2=0$, then \eqref{eqn:equation} implies $a_1\ne0$. Clearly, $\vp$ is constant.
		\end{enumerate}
		The proof is complete.	
	\end{proof}

	We conclude this section with the following theorem, which provides a criterion for a composition operator to be bounded on $\HbE$ in case $\lam_1=0$.
	
	\begin{thm}[Criterion for $\lam_1=0$]\label{T:bounded2}
		Suppose $\lam_1=0$. Let $\vp$ is an entire function and $\Cp$ be the induced composition operator. Then $\Cp$ is bounded on $\HbE$ if and only if one of the following cases happen
		\begin{enumerate}[(i)]
			\item $\vp$ is constant, or
			\item $\vp(z)=z+b$, for some $b\in\C$ with $\Re(b)\ge 0$.
		\end{enumerate}
		Moreover, $\|\Cp\|\ge 1$ in Case $(i)$, and $\|\Cp\|=1$ in Case $(ii)$.
	\end{thm}
	
	\begin{proof}	
		The sufficiency is proved in Proposition \ref{P:suf}, and Proposition \ref{P:lam=0} establishes the necessity, so $\vp$ is either constant or of the affine form $z+b$ with $\Re(b)\ge0$. For the norm estimation of $\Cp$, following Claim 3 of Theorem \ref{T:bounded1}, we obtain $\|\Cp\|\ge \|\Cp q_1\|=1$. This is true for both cases $(i)$ and $(ii)$. In addition, in Case $(ii)$, if $f(z)=\sumn a_ne^{-\lam_nz}\in\HbE$ is nonzero, as $0\le\lam_n\uparrow+\infty$ and $\Re(b)\ge 0$, we have
		\[
			\|\Cp f\|^2=\left\|\sumn a_ne^{-\lam_n(z+b)}\right\|^2=\sumn |a_n|^2\be_n^2e^{-2\lam_n\Re(b)} \le \sumn |a_n|^2\be_n^2 = \|f\|^2,
		\]
		so $\|\Cp\|\le 1$. Hence $\|\Cp\|=1$ for Case $(ii)$.
	\end{proof}

	\section{Concluding remarks}\label{sec:conclusion}
	
	In the present paper, we study the relation between an entire function $\vp$ and the boundedness of the induced composition operator $\Cp$ acting on spaces of entire Dirichlet series $\HbE$. We generalize the result of bounded operators $\Cp$ on spaces $\HbS$ and include the untreated case $\lam_1=0$.
	
	The following theorem establishes the \emph{complete characterization} of the boundedness of $\Cp$, which shows that the criteria do not depend on whether the weight sequence $\be=(\be_n)$ satisfies condition $(E)$ or any condition stronger than $(E)$, such as $(S)$.
	
	\begin{thm}[Criterion for any space $\HbE$] \label{T:conclusion} Let $\be$ be a sequence of positive real number with condition $(E)$, and $\vp$ be an entire function. Consider the following statements.
		\begin{enumerate}[(i)]
			\item $\vp$ is constant,
			\item $\vp(z)=z+b$ for some $b\in\C,\ \Re(b)\ge0$.
		\end{enumerate}
		
		The following are true about the boundedness of the composition operator $\Cp$ acting on the induced Hilbert space $\HbE$:
		\begin{enumerate}[(1)]
			\item If $\lam_1=0$, then $\Cp$ is bounded if and only if exactly one of conditions (i) or (ii) holds.
			\item If $\lam_1>0$, then $\Cp$ is bounded if and only if (ii) holds.
		\end{enumerate}
	
		Furthermore, in case $(ii)$, the operator norm is given by $\|\Cp\|=e^{\lam_1\Re(b)}$.
	\end{thm}

	This theorem comes from Propositions \ref{P:suf}, and Theorems \ref{T:bounded1} and \ref{T:bounded2}.
	
	\smallskip
	
	Since the proofs of criteria for the compactness, compact difference, Hilbert--Schmidtness, cyclicity, etc. of composition operators $\Cp$ acting on $\HbS$ in \cite{HHK,WY,HuK12} do not directly use condition $(S)$ but the necessary condition $\vp(z)=z+b$ with $\Re(b)\ge0$, these result may still be true for the general spaces $\HbE$, with the exception that $\vp$ being constants is allowed for the case $\lam_1=0$.
	
	Other findings, such as norm estimation through reproducing kernels in \cite{WY}, which directly uses $(S)$ in their computation, need to be reconsidered when working with condition $(E)$. However, we hope that our discovery and method may inspire readers to investigate further these problems in the future.


\bigskip


\begin{thebibliography}{99}

\bibitem{A76} T.M. Apostol, \textit{Introduction to Analytic Number Theory}, Springer Berlin Heidelberg, 1976.

\bibitem{Ar50} N. Aronszajn, Theory of reproducing kernels, \textit{Trans. Amer. Math. Soc.} 68 (1950), 337-–404.

\bibitem{CM} C.C. Cowen and B.D. MacCluer, \textit{Composition Operators on Spaces of Analytic Functions}, CRC Press, Boca Raton, 1995.

\bibitem{FQ04} C. Finet and H. Que\'elec, Numerical range of composition operators on a Hilbert space of Dirichlet series, \textit{Linear Algebra Appl.} 377 (2004), 1–-10.

\bibitem{FQV04} C. Finet, H. Que\'elec, and A. Volberg, Compactness of composition operators on a Hilbert space of Dirichlet series, \textit{J. Funct. Anal.} 211 (2004), 271–-287.

\bibitem{GH99} J. Gordon and H. Hedenmalm, The composition operators on the space of Dirichlet series with square summable coeﬃcients, \textit{Michigan Math. J.} 46 (1999), no. 2, 313-–329.

\bibitem{HR15} G.H. Hardy and M. Riesz, \textit{The General Theory of Dirichlet’s Series}, 1st ed., Cambridge University Press, 1915.

\bibitem{HHK} X. Hou, B. Hu, and L.H. Khoi, Hilbert spaces of entire Dirichlet series and composition operators, \textit{J. Math. Anal. Appl.} 401 (2013), no. 1, 416–-429.

\bibitem{HK12} X. Hou and L.H. Khoi, Some properties of composition operators on entire Dirichlet series with real frequencies, \textit{C. R. Math. Acad. Sci. Paris} 350 (2012), no. 3–4, 149-–152.

\bibitem{HuK12} B. Hu and L.H. Khoi, Numerical range of composition operators on Hilbert spaces of entire Dirichlet series, \textit{Interactions between real and complex analysis}, Hanoi, 2012, 285-–299.

\bibitem{L83} A.F. Leont'ev, \textit{Entire Functions. Series of Exponentials}, Nauka, Moscow, 1983 (Russian).

\bibitem{M69} S. Mandelbrojt, \textit{S\'eries de Dirichlet. Principes et M\'ethodes}, Monographies internationales de math\'ematiques modernes, 1969 (French).

\bibitem{Pol} G. P´olya, On an integral function of an integral function, \textit{J. Lond. Math. Soc.} 1 (1926), 12–-15.

\bibitem{Red66} A.R. Reddy, On entire Dirichlet series of zero order, \textit{T\^oh\^oku Math. J.} 18 (1966), no. 2, 144–-155.

\bibitem{R28} J. F. Ritt, On certain points in the theory of Dirichlet series, \textit{Amer. J. Math.} 50 (1928), no. 1, 73–-86.

\bibitem{WY} M. Wang and X. Yao, Some properties of composition operators on Hilbert spaces of Dirichlet series, \textit{Complex Var. Elliptic Equ.} 60 (2015), no. 7, 992–-1004.

\bibitem{Y17} X. Yao, Complex symmetric composition operators on a Hilbert space of Dirichlet series, \textit{J. Math. Anal. Appl.} 452 (2017), no. 2, 1413–-1419.

\end{thebibliography}
\end{document}